\documentclass{amsart}

\newtheorem{theorem}{Theorem}[section]
\newtheorem{lemma}[theorem]{Lemma}

\theoremstyle{definition}
\newtheorem{definition}[theorem]{Definition}

\theoremstyle{remark}

\newcommand{\ld}{\lambda}

\newcommand{\De}{\Delta}
\newcommand{\nb}{\nabla}

\newcommand{\af}{\alpha}

\newcommand{\ba}{\begin{array}}
\newcommand{\ea}{\end{array}}
\newcommand{\be}{\begin{equation}}
\newcommand{\ee}{\end{equation}}
\newcommand{\ban}{\begin{eqnarray*}}
\newcommand{\ean}{\end{eqnarray*}}

\newcommand{\C}{\mathbb{C}}
\newcommand{\R}{\mathbb{R}}

\newcommand{\Z}{\mathbb{Z}}

\newcommand{\I}{\mathcal {I}}
\newcommand{\J}{\mathcal {J}}

\numberwithin{equation}{section}

\begin{document}

\title[Deficiency  indices of difference equations ]
{The positive and negative  deficiency  indices of formally self-adjoint difference equations}

\author{GUOJING REN}

\address{School of Mathematics and Quantitative Economics,
 Shandong  University of Finance and Economics, 250014, P. R.
China}

\email{gjren@sdufe.edu.cn}

\thanks{This work is supported by the  NSF of Shandong Province, P.R. China
[grant numbers ZR2020MA012].}

\subjclass[2000]{39A70, 47B39, 34B20.}

\keywords{ Positive and negative  deficiency  indices, Hermitian operator,
Formally self-adjoint, Difference equation, $K$-real.}

\baselineskip=17pt

\begin{abstract}
This paper is concerned with formally self-adjoint difference equations
and their positive and negative  deficiency  indices.
It is shown that the order of any formally self-adjoint difference equation is even,
and some characterizations of formally self-adjoint difference equations are established.
Further, we show that the positive and negative  deficiency  indices  are always equal,
which implies  the existence of  the self-adjoint extensions of the minimal linear relations
generated by the difference equations.
This is an important and essential difference  between formally self-adjoint difference equations and their
corresponding differential equations in the spectral theory.
\end{abstract}
\maketitle

\section{Introduction}

Difference equations are usually regarded as the discretization of the corresponding differential equations.
According to the existing results, most of the  properties in spectral theory of difference equations
coincide with those of the corresponding differential equations;
only a few, but important, are different.
It has been found that the  maximal operator corresponding to a formally self-adjoint difference equation
may be multi-valued, and the corresponding minimal operator  may be non-densely defined \cite{Ren5, Shi2}.
Therefore, the classical spectral theory for symmetric operators, i.e.,
densely defined and Hermitian single-valued operators, are not available  in the studying
of the spectral properties of  difference equations in general.
Due to this reason, some researchers focus on extending the spectral theory of linear operators to
linear  non-densely defined or multi-valued operators (which are called linear relations or linear subspaces),
and many good results have been obtained
(See \cite{Coddington1,Cross,Shi3,Shi4} and their references).

According to the generalized von Neumann theory  and the GKN theory,
a Hermitian linear relation, has a self-adjoint extension
if and only if its positive and negative deficiency indices are equal,
and in this case the  domain of the self-adjoint extension has a close relationship
with the deficiency indices \cite{Coddington1,Shi3}.
So, the positive and negative deficiency indices of Hermitian linear relations play
a key role in the study of spectral theory of linear relations.

Now, we briefly recall some important results about the deficiency indices of
formally self-adjoint differential  and difference equations, respectively.

The theory of  positive and negative deficiency indices of Hermitian differential
equations has been well developed.
The canonical form
of any formally self-adjoint differential expression $\tau$
of order $m$ on some interval $\mathcal{J}$ is given by
\begin{align}
\tau y:=\sum_{j=0}^{[m/2]}(-1)^j[a_jy^{(j)}]^{(j)}+i\sum_{k=0}^{[(m-1)/2]}
\left[\left(b_ky^{(k+1)}\right)^{(k)}+\left(b_ky^{(k)}\right)^{(k+1)}\right],
\end{align}
where $a_j$ and $b_k$ are all real-valued functions belonging to $C^{\infty}(\mathcal{J})$, and $i=\sqrt{-1}$.
The following equation
\begin{align}
\tau y(x)=\ld w(x)y(x),\quad x\in \J
\end{align}
is called the  formally self-adjoint  differential equation corresponding to $\tau$ \cite[XIII,2]{Dunford&Schwartz},
where $w\ge 0$ is called the weighted function, $\ld\in \C$ is the spectral parameter.
In the case that both endpoints of $\J$ are singular,
we can divided $\J$ into two subintervals such that  each subinterval has at least one regular endpoint.
So we  take $\J:=[0, +\infty)$ without loss generality.

Since $(1.1)$ is formally self-adjoint, the minimal linear relation
generated by (1.2) in the corresponding Hilbert space is Hermitian,
and consequently its deficiency index $d_{\ld}(\tau)$  is constant when
$\ld$ is in the upper  half-plane and lower half-plane.
Denoted $d_\pm(\tau):=d_{\pm i}(\tau)$, which are called
the positive and negative deficiency indices  of  (1.2).

In addition, by $n_{\ld}(\tau)$ denote the number of linearly independent solutions of (1.2)
which satisfies
\begin{equation*}
  \int_{\J}w(x)|y(x)|^2dx<+\infty.
\end{equation*}
and denote $n_{\pm}(\tau):= n_{\pm i}(\tau)$. The definiteness condition for (1.2) is given by:
\begin{itemize}
 \item[($A_1$)] there exists a bounded interval $\J_0\subset\J$
such that for any $\ld\in\C$ and for any non-trivial solution $y$ of (1.2), the following always holds
\begin{equation*}
  \int_{\J_0}w(x)|y(x)|^2dx>0.
\end{equation*}
\end{itemize}
It is evident that $(A_1)$ holds when $w(t)> 0$ on $\J$.
Under the assumption $(A_1)$,
it has been shown that $d_{\ld}(\tau)=n_{\ld}(\tau)$ for all $\ld\in \C$.
This equivalence does not hold without the assumption $(A_1)$ \cite[Proposition 2.19]{Lesch}).

In the case that all the coefficients of $\tau$ are real-valued, i.e., all $b_k(t)\equiv 0$ on $\J$,
it has been easily shown that $n_+(\tau)=n_-(\tau)$.
This, together with $(A_1)$, implies  $ d_+(\tau)=d_-(\tau)$.
However, the values of $d_\pm(\tau)$  may differ when $\tau$ has complex-valued coefficients.
Mcleod \cite{Mcleod}  first gave an  example of a fourth-order formally self-adjoint
differential equation with  $d_+(\tau)=3$ and $d_-(\tau)=2$.
Later, Kogan and Rofe-Beketov \cite{Kogan1,Kogan2} showed that
\begin{equation*}
  |d_+(\tau)-d_-(\tau)|=1
\end{equation*}
happens for any $m\ge 3$.

Many authors  are interested in the  positive and negative
deficiency indices of formally self-adjoint difference equations,
and have got many excellent results.
A  formally self-adjoint difference expression  is necessary to be even,
and it has the following form (see Theorem 3.3):
\begin{align}
\mathcal{L}y:=\sum_{j=0}^{n}F^{j}(A_jy)+\sum_{j=1}^{n}\overline{A}_jF^{-j}y,
\end{align}
where $A_j$, $j=1,\ldots,n$, are complex-valued functions, and $A_0$ is real-valued on $\I$;
$A_n\neq 0$ on $\I$; $F$ is the forward shift operator, i.e., $Fy(t)=y(t+1)$.
The following equation
\begin{align}
(\mathcal{L}y)(t)=\ld w(t) y(t),\quad t\in \I
\end{align}
is called the formally self-adjoint equation generated by $\mathcal{L}$,
where  $w(t)\ge 0$ and $\ld\in \C$.
We take $\I:=\{t\}_{t=0}^{ +\infty}$ in the following.
It is worth noting that  the order of any formally self-adjoint difference equation is  even.
This is a difference between formally self-adjoint difference equations and  differential equations.

In addition,  $\mathcal{L}$ defined by (1.3) can be rewritten as:
\begin{align}
\mathcal{L}y=\sum_{j=0}^{n}(-1)^j\Delta^j(p_j\nb^jy)+i\sum_{k=1}^{n}[(-1)^{k+1}\Delta^k(q_ky)+q_k\nb^ky],
\end{align}
where  $\De$ and $\nb$ are the forward  and backward difference
operators, respectively, i.e., $\De y(t)=y(t+1)-y(t)$ and $\nb
y(t)=y(t)-y(t-1)$; the coefficients $p_j$ and $q_k$ are all real-valued on $\I$.
The coefficients between (1.3) and (1.5) have the following relationship:
\begin{align*}
&A_0(t)=P_0(t), \quad A_j(t)=P_j(t)+iQ_j(t),\quad 1\leq j\leq n,\\
&P_j(t)=(-1)^{j}\sum_{s=j}^{n}\sum_{k=0}^{s-j}C_{s}^{k}C_{s}^{s-j-k}p_s(t+k),    \quad 0\leq j\leq n,        \\
&Q_j(t)=(-1)^{j+1}\sum_{k=j}^{n}C_{k}^{j}q_k(t),\quad 1\leq j\leq n,
\end{align*}
where $C_{s}^{k}=\frac{s!}{k!(s-k)!}$.

The definiteness condition for (1.4) is given by:
\begin{itemize}
 \item[($A_2$)] there exists a bounded integer interval $\I_0\subset\I$
such that for any $\ld\in\C$ and for any non-trivial solution $y$ of (1.4), the following always holds
\begin{equation*}
  \sum_{t\in \I_0}w(t)|y(t)|^2>0.
\end{equation*}
\end{itemize}

Similarly as the continuous case,  the minimal linear relation generated by (1.4)
is Hermitian and consequently its deficiency index $d_{\ld}(\mathcal{L})$
is constant when $\ld$ is in the upper and lower half-planes \cite{Ren5}.
Denote $d_{\pm}(\mathcal{L}):=d_{\pm i}(\mathcal{L})$.
Under the assumption $(A_2)$, $d_{\ld}(\mathcal{L})=n_{\ld}(\mathcal{L})$ holds,
where $n_{\ld}(\mathcal{L})$ is the number of linearly independent  solutions of (1.4), which satisfies
\begin{equation*}
  \sum_{t\in \I}w(t)|y(t)|^2<+\infty.
\end{equation*}
Furthermore,  $n\le d_+(\mathcal{L})=d_-(\mathcal{L})\le 2n$ holds and all the values in this range can be realized.
The readers are referred to \cite{Atkinson1,Chen1, Clark1, Jirari1,Ren1,Ren2, Sun1} for more details.

Up to now, all the existing results on the the positive and negative deficiency indices of
formally self-adjoint difference equations  coincide with those of their corresponding differential equations.
So, one may think there would exist  some special case of (1.4),
corresponding to Mcleod's example mentioned above, satisfying  $d_+(\mathcal{L})\neq d_-(\mathcal{L})$.
In this manuscript, we show that  the positive and negative deficiency indices of
any formally self-adjoint difference equations are equal (Theorem 4.1).
This is a new and  important difference in the spectral theory between
the formally self-adjoint difference equations  and their corresponding differential equations.

The rest of the paper is organized as follows.
In Section 2, we introduce some basic concepts  of the spectral theory in Hilbert space and give
some  sufficient and necessary conditions for an Hermitian operator to have equal  positive and negative deficiency indices.
In Section 3, we establish two characterizations of formally self-adjoint difference expressions.
In Section 4, we prove  the positive and negative deficiency indices of
any formally self-adjoint difference equations are equal (Theorem 4.1).

\section{Fundamental results about linear relations and linear operators}

In this section, we introduce some basic concepts of linear relations and linear operators,
and show  some results on the deficiency indices of Hermitian  linear relations and
linear operators and their self-adjoint extensions.


By $\C$, $\R$, and $\Z$ denote the sets of all complex numbers, real numbers and integers, separately.
By $\bar{a}$ denote the conjugation of  $a\in \C$, and  $i:=\sqrt{-1}$.

Let $X$  be a complex Hilbert  space  with inner product $\langle\cdot, \cdot\rangle$.
Let $T_1$ and $T_2$ be  two linear relations (briefly, relation) in $X^2:=X\times X$.
Denote
\begin{eqnarray*}
&&\mathcal{D}(T):=\{x\in X:\; (x,f)\in T {\;\rm for \; some}\; f\in X\},\\
&&\mathcal{R}(T):= \{f\in X: \;(x,f)\in T {\;\rm for \; some}\; x\in X\},\\
&&\mathcal{N}(T):=\{x\in X:\;(x,0)\in T\},\\
&&T^*:=\{(x,f)\in X^2: \;\langle x,g\rangle=\langle f,y\rangle\; {\rm  for\; all}\; (y,g)\in T\},\\
&&T(x):=\{ f\in X:\; (x, f)\in T\},\\
&&T_1+T_2:=\{(x,f_1+f_2)\in X^2: \;(x,f_i)\in T_i, i=1,2\}.
\end{eqnarray*}
It is clear that $T(0)=\{0\}$ if and only if there exists a unique linear  operator, denoted by
$S_T:\mathcal{D}(T)\to X$, such that its graph $G(S_T)=T$. In this case, $S_Tx=f$ for any $(x,f)\in T$.

In addition, $T$ is said to be Hermitian  if $T\subset T^*$;
$T$ is said to be symmetric   if $T\subset T^*$ and $\mathcal{D}(T)$ is dense in $X$;
$T$ is said to be self-adjoint  if $T=T^*$.


If  $T$ is a closed linear relation in $X^2$, then $T$ can be decomposed as   \cite{Arens}:
\begin{eqnarray*}
T=T_{s}\oplus T_{\infty},
\end{eqnarray*}
where $T_{\infty}=\{(0,f)\in X^2:\; f\in T(0)\}$ is called the pure multi-valued parts of $T$, and
\begin{eqnarray*}
T_{s}=T\ominus T_{\infty}
\end{eqnarray*}
is called the operator part of $T$.

\begin{lemma}  \cite[Proposition 2.1, 3.1]{Shi4}
Let $T$ be a closed Hermitian relation in  $X^2$. Then \begin{align*}
T_{\infty}=T\cap (T(0))^2,\quad T_{s}=T\cap (T(0)^\bot)^2,
\end{align*}
and  $T_s$ is  closed Hermitian  in $(T(0)^\bot)^2$.
Further, if $T$ is self-adjoint  in $X^2$, then $T_s$ is
 self-adjoint   in $(T(0)^\bot)^2$.
\end{lemma}

Let $T$ be  a linear relation (operator)  in $X^2$ ($X$).
The subspace $\mathcal{R}(T-\ld I))^\bot$ and the number
$d_{\ld}(T) := \dim \mathcal{R}(T-\ld I)^\bot$ are called the deficiency
space and deficiency index of $T$ and $\ld$, separately.
By $\overline{T}$ denote the closure of $T$.
It can be easily verified that $d_{\ld}(T)=d_{\ld}(\overline{T})$ for all $\ld\in \C$.

Let $T$ be  a linear relation  in $X^2$. The set
\begin{equation*}
  \Gamma(T):= \{\ld\in \C: \exists\, c(\ld) > 0 \;{\rm s.t.}\; \|f-\ld x\|\ge c(\ld)\|x\|,
  \forall (x,f)\in T\}
\end{equation*}
is called the regularity domain of $T$.

It has been shown by  \cite[Theorem 2.3]{Shi3} that the deficiency index $d_{\ld}(T)$
is constant in each connected subset of $\Gamma(T)$.
If $T$ is Hermitian, then $d_{\ld}(T)$ is constant in the upper and lower half-planes.
So, we can denote $d_{\pm}(T) := d_{\pm i}(T)$ for an  Hermitian linear relation $T$,
and call $d_{\pm}(T)$ the positive and negative deficiency indices of  $T$, respectively.

\begin{lemma} \cite[Theorem 15]{Coddington1}
Let $T$ be a closed Hermitian relation in  $X^2$. Then $T$ has self-adjoint  extensions if and only if
$d_+(T)=d_-(T)$.
\end{lemma}

\begin{lemma} \cite[Corollary 2.1]{Shi4} Let $T$ be a closed Hermitian relation in  $X^2$.
Then $d_{\pm}(T)=d_{\pm}(T_s)$, and consequently,
 $T$ has a self-adjoint  extension in
$X^2$ if and only if $T_s$ has a self-adjoint  extension in $(T(0)^\bot)^2$ .
\end{lemma}

A  relation $T$  is said to be bounded from below (above) if there exists a number $c\in \R$ such that
\begin{equation*}
  \langle x,f\rangle\ge c\|x\|^2, \quad  (\langle x,f\rangle\le c\|x\|^2), \quad  \forall (x,f)\in T,
\end{equation*}
while such a constant $c$ is called a lower (upper) bound of $T$.

\begin{lemma}\cite[Proposition 1.4.6]{Behrndt}
Let $T$ be an Hermitian relation and be bounded from below with lower bound $c$. Then
$\C\setminus [c,+\infty) \subset \Gamma(T)$, and the deficiency index $d_{\ld}(T)$ is constant for all
$\ld\in \C\setminus [c,+\infty)$.
\end{lemma}

Combining Lemmas 2.2 and 2.4, one can get the following result.

\begin{lemma}
Let $T$ be an Hermitian relation and  be  bounded from below. Then
$d_+(T)=d_-(T)$ and $T$ has self-adjoint extensions.
\end{lemma}

Next, we pay attention to  the self-adjointness of  an operator by using the concept of $K$-real.
Let $S$ be an operator on $X$.

\begin{definition}\cite[Section 8.1]{Weidmann1}
Let $X$ be a complex Hilbert space. A mapping $K$ of $X$ onto itself is called a conjugation if
\begin{enumerate}
  \item  $K(ax+by)=\bar{a}K(x)+\bar{b}K(y)$ for all $x,y\in X$, $a,b\in\C$.
  \item $K^2=I$.
  \item $\langle Kx,Ky\rangle=\langle y,x\rangle$ for all $x,y\in X$.
\end{enumerate}
An  operator $S$ on $X$ is said to be a $K$-real if
\begin{enumerate}
  \item  $K\mathcal{D}(S)\subset \mathcal{D}(S)$
  \item $SKx=KSx$ for $x\in \mathcal{D}(S)$.
 \end{enumerate}
\end{definition}

A sufficient condition for a symmetric  operator
to have equal positive and negative deficiency indices is given by \cite[Theorem 8.9]{Weidmann1}.
Since $S$  is not required to be densely defined in the proof,
the assertion  is still true for Hermitian   operators.

\begin{lemma}\cite[Theorem 8.9]{Weidmann1}
Let $X$ be a complex Hilbert space, and let $K$ be a conjugation on $X$.
If $S$ is a $K$-real Hermitian  operator on $X$, then $d_+(S)=d_-(S)$.
\end{lemma}

The following is  a necessary condition for an operator $S$ to  be self-adjoint.

\begin{theorem}
Let $X$ be a complex Hilbert space, and let $S$ be a self-adjoint operator  on $X$.
Then there exists  a conjugation $K$ for which $S$ is $K$-real.
\end{theorem}

To prove Theorem 2.8, we introduce several  notations and  a lemma first.
The readers are referred to \cite{Weidmann1}.

Let $\{\rho_{\af}:\af\in A\}$ be a family of right continuous non-decreasing function defined on $\R$.
By $L^2(\R,\rho_\af)$ denote the set of all the square integrable functions with respect to $\rho_\af$ on $\R$.
It has been shown that $L^2(\R,\rho_\af)$ is a Hilbert space with inner
\begin{equation*}
 \langle f,g\rangle_{\af} =\int_{\R}\bar{g}(t)f(t) d\rho_\af(t).
\end{equation*}
By $\oplus_{\af\in A}L^2(\R,\rho_{\af})$ denote the orthogonal sum of the spaces $L^2(\R,\rho_\af)$.
Then $\oplus_{\af\in A}L^2(\R,\rho_{\af})$ is a Hilbert space with inner
\begin{equation*}
  \langle (x_\af), (y_\af)\rangle=\sum_{\af\in A}\langle x_\af,y_\af\rangle_\af,\quad (x_\af),(y_\af)\in \oplus_{\af\in A}L^2(\R,\rho_{\af}).
\end{equation*}
By $K_0$ denote the natural conjugation on $\oplus_{\af\in A}L^2(\R,\rho_{\af})$, i.e.,
\begin{equation}
  K_0(x_\af):=(\bar{x}_\af),\quad (x_\af)\in \oplus_{\af\in A}L^2(\R,\rho_{\af}).
\end{equation}

The following result is a combination of  Theorems 7.16-7.18 of \cite{Weidmann1}.

\begin{lemma}
Let $S$ be a self-adjoint  operator on $X$. Then there exists exactly one spectral  family $E$ for which
$T=\hat{E}(\rm id)$. Moreover, for this spectral  family $E$, there exists a family    $\{\rho_\af:\af\in A\}$
of right continuous non-decreasing functions ( the cardinality of $A$ is at most the dimension of $X$)
and a unitary operator $U:X\to \oplus_{\af\in A}L^2(\R,\rho_{\af})$ for which

\begin{equation}
 S=U^{-1}S_{\rm id}U,
\end{equation}
where $S_{\rm id}$ denotes the maximal operator of multiplication by the function $\rm id$ on $\oplus_{\af\in A}L^2(\R,\rho_{\af})$.
\end{lemma}

Now we give the proof of Theorem 2.8.

\begin{proof}
Let $S$ be a self-adjoint operator  on $X$. Then it follows from Lemma 2.9 that
there exists a  family $\{\rho_\af:\af\in A\}$
of right continuous non-decreasing functions (the cardinality of $A$ is at most the dimension of $X$)
and a unitary operator $U:X\to \oplus_{\af\in A}L^2(\R,\rho_{\af})$ such that (2.2) holds.
Based on the proof of  \cite[Theorems 7.16]{Weidmann1},
the unitary operator $U$ can be constructed  to be linear, i.e.,
\begin{equation*}
  U(ax+by)=aU(x)+bU(y), \quad x,y\in X,\;a,b\in \C.
\end{equation*}
Let
\begin{equation}
  K=U^{-1}K_0U,
\end{equation}
where   $K_0$ is  the natural conjugation on $\oplus_{\af\in A}L^2(\R,\rho_{\af})$ defined by (2.1).
Then it follows that
\begin{equation}
  K(ax+by)=\bar{a}K(x)+\bar{b}K(y)
\end{equation}
for all $x,y\in X$, $a,b\in\C$.
Further,  by the facts that $K_0$ is  a conjugation and $U$ is a unitary operator, it follows that
\begin{align}
&K^2=U^{-1}K_0UU^{-1}K_0U=I,\\\nonumber
&\langle Kx,Ky\rangle=\langle K_0Ux,K_0Uy\rangle=\langle Uy,Ux\rangle=\langle y,x\rangle,\quad x,y\in X.
\end{align}
These yield that  $K$ defined by (2.3) is a conjugation on $X$.

Further, it can be verified  that $S_{\rm id}$ is $K_0$-real, and $U\mathcal{D}(S)\subset \mathcal{D}(S_{\rm id})$.
Therefore it follows that  $K\mathcal{D}(S)\subset \mathcal{D}(S)$, and
\begin{equation*}
  SKx=U^{-1}S_{\rm id}K_0Ux=U^{-1}K_0S_{\rm id}Ux=KSx, \quad x\in \mathcal{D}(S).
\end{equation*}
This implies that $S$ is $K$-real.
The proof is complete.
\end{proof}

\section{Characterization  of formally self-adjoint difference expressions}

In this section  we pay attention to the characterization  of formally self-adjoint difference expressions.

First, we introduce some notations and concepts.
Let $X$ be a vector space over $\C$. A mapping $s:\;X\times X \to \C$ is called a sesquilinear form on $X$ if it follows
\begin{align*}
 & s[x,ay+bz]=\bar{a}s[x,y]+\bar{b}s[x,z], \\
 & s[ay+bz,x]=as[y,x]+bs[z,x],
\end{align*}
for all $x,y,z\in X$ and $a,b\in \C$.
Let $\I=\{t\}_{t=0}^{+\infty}$. For some $n\in\Z$, denote
\begin{equation*}
  l(\I,n)=\{x=\{x(t)\}_{t=-n}^{+\infty}:\;x(t)\in \C\}
\end{equation*}

Let $L$ be an arbitrary   difference expression of order $m$  defined on $l(\I,m)$:
\begin{align*}
    (Ly)(t):=\sum_{j=0}^kA_j(t+j)y(t+j)+\sum_{j=1}^sA_{-j}(t)y(t-j),
\end{align*}
where all  $A_j(t)$, $-s\le j\le k$, are complex valued functions;
$k$ and $s$ are non-negative integer satisfying $k+s=m$, and $A_{k}(t+k)\neq 0$, $A_{-s}(t)\ne 0$ for   $t\in \I$.
By $F$ denote the forward shift operator, i.e.,
\begin{equation*}
 Fy(t)=y(t+1).
\end{equation*}
and define   $F^{-1}y(t)=y(t-1)$,  $F^{j}=FF^{j-1}$.
Thus $L$ can be rewritten briefly as
\begin{align}
    Ly=\sum_{j=0}^kF^j(A_jy)+\sum_{j=1}^sA_{-j}F^{-j}y.
\end{align}

Now, we give the definition of the  formal adjoint of $L$ and try to establish the characterization  of its formal adjoint.

\begin{definition} Let $L$ be an $m$th-order  difference expression.
An $m$th-order difference expression, denoted by $L^+$, is said to be a formal adjoint of $L$,  if
there exists a sesquilinear form $s[\cdot, \cdot]$ on $l(\I,m)$ such that
\begin{equation*}
\bar{y}(t)(Lx)(t)-\overline{(L^+y)(t)}x(t)=\Delta s[x,y](t),\quad t\in \I.
\end{equation*}
Moreover, $L$ is said to be formally self-adjoint if $L=L^+$.
\end{definition}

For $n\ge 0$, let $L_n$ be a  difference expression   defined as the following:
\begin{align}
    (L_ny)(t):=\sum_{j=0}^nF^{j}(B_jy)(t),
\end{align}
where all $B_j(t)$, $0\le j\le n$, are complex-valued functions, $B_n(t)\ne 0$ on $\I$.
We get the following result on the formal adjoint of $L_n$.

\begin{lemma}
Let $L_n$ be an $n$th-order difference expression defined by $(3.2)$. Then

\begin{enumerate}
\item (Green Formula) For any $x,y\in l(\I,n)$,  $0\le k\leq r$,
\begin{align}
\sum_{t=k}^{r}\left[\bar{y}(t)L_nx(t)- \overline{(L_n^+y(t))}x(t)\right]=  s[x,y](r+1)- s[x,y](k),
\end{align}
where  $L_n^+$ is given by
\begin{align}
  (L_n^+y)(t)=\sum_{j=0}^{n}\overline{B}_{j}(t)F^{-j}y(t),
\end{align}
and $s[x,y]$ is a sesquilinear on $l(\I,n)$ with the form
\begin{align}
s[x,y](t)=\sum_{k=1}^n\sum_{j=0}^{k-1}\bar{y}(t-j-1)B_k(t+j)x(t+j).
\end{align}
\item  $L_n+L_n^+$ is a $2n$th-order formally self-adjoint difference expression.
\end{enumerate}
\end{lemma}

\begin{proof}
(1) Let $L_n^+$ and $s[x,y]$ be defined as (3.4) and (3.5), respectively. It suffices to  show
\begin{align}
\bar{y}(t)L_nx(t)- \overline{(L_n^+y)(t)}x(t)=\De s[x,y](t),\quad t\in \I.
\end{align}
We will proof (3.6)  by induction. For $n=1$, it follows that
\begin{align*}
\bar{y}(t)(L_1x)(t)- \overline{(L_1^+y)(t)}x(t)&=\bar{y}(t)B_1(t+1)x(t+1)-\bar{y}(t-1)B_1(t)x(t)\\
&= \De [\bar{y}(t-1)B_1(t)x(t)].
\end{align*}
This yields  (3.6) with $n=1$.
Assume that (3.6) hold for  $n-1$.
It can be easily verified
\begin{equation*}
 \bar{y}(t)B_n(t+n)x(t+n)-\bar{y}(t-n)B_n(t)x(t)=\De\sum_{j=0}^{n-1}[\bar{y}(t-j-1)B_n(t+j)x(t+j)].
\end{equation*}
Then it follows that
\begin{align*}
    &\bar{y}(t)(L_nx)(t)- \overline{(L_n^+y)(t)}x(t)\\
=&[\bar{y}(t)(L_{n-1}x)(t)-\overline{(L_{n-1}^+y)(t)}x(t)]+ \bar{y}(t)B_n(t+n)x(t+n)-\bar{y}(t-n)B_n(t)x(t)\\
=&\Delta\sum_{k=1}^{n-1}\sum_{j=0}^{k-1}\bar{y}(t-j-1)B_k(t+j)x(t+j)+ \De\sum_{j=0}^{n-1}[\bar{y}(t-j-1)B_n(t+j)x(t+j)]\\
=&\Delta\sum_{k=1}^{n}\sum_{j=0}^{k-1}\bar{y}(t-j-1)B_k(t+j)x(t+j).
\end{align*}
Thus (3.6), and consequently, assertion (1) is proved.

(2) Assertion (1) yields  that $(L_n^+)^+=L_n$.
So,  $(L_n+L_n^+)^+=L_n+L_n^+$, which implies  that $L_n+L_n^+$ is formally self-adjoint.
The proof is complete.
\end{proof}

Based on the above discussion,
we now establish a characterization of  formally self-adjoint difference expressions.

\begin{theorem}
Let   $\mathcal{L}$ be a formally self-adjoint difference expression of order $m$.
Then   $m$ is even, saying $m=2n$,  and   $\mathcal{L}$ has the the following form
\begin{align}
\mathcal{L}y=\sum_{j=0}^{n}F^{j}(A_jy)+\sum_{j=1}^{n}\overline{A}_jF^{-j}y,
\end{align}
where $A_j$, $j=1,\ldots,n$, are complex-valued functions, and $A_0$ is real-valued.
\end{theorem}

\begin{proof} First, we shown $\mathcal{L}$ defined by (3.7) is self-adjoint.
Denote
\begin{equation*}
 L_ny=\sum_{j=1}^{n}F^j(A_jy)+\frac{1}{2}A_0y.
\end{equation*}
Then one has that  $\mathcal{L}=L_n+L_n^+$  by (1) of Lemma 3.2, and consequently,
$\mathcal{L}$ is formally self-adjoint by (2) of Lemma 3.2.

On the other hand, we proof $\mathcal{L}$ has form (3.7) if $\mathcal{L}$ is self-adjoint.
Let  $L$  be any $m$th-order difference expression with the form of  (3.1).
Then it can be written as $L=L_k+L_s$,
where
\begin{align*}
    (L_ky)(t)=\sum_{j=0}^kA_j(t+j)y(t+j),\quad L_sy)(t)=\sum_{j=1}^sA_{-j}(t)y(t-j).
\end{align*}
It follows from (1) of Lemma 3.2 that
\begin{align*}
   L^+y(t)=L_k^+y(t)+\hat{L}_s^+y(t)=\sum_{j=1}^s\overline{A}_{-j}(t+j)y(t+j)+ \sum_{j=0}^k\overline{A}_j(t)y(t-j).
\end{align*}
This yields that $L=L^+$  if and only if $s=k$, which implies that $m$ is even, saying $m=2n$; and
$A_j(t)\equiv\overline{A}_{-j}(t)$ for  $j=0,1,\ldots,n$,
which implies that  $A_0(t)$ is real-valued. The proof is complete.
\end{proof}

\begin{theorem}
For any $n$th-order forward difference expression $L_n$,
$L_nL_n^+$ is a   formally self-adjoint difference expression of order $2n$.
\end{theorem}

\begin{proof}
Let  $L_n$ be defined as (3.2). Then  it can be verified that
\begin{align*}
(L_nL_n^+)y=L_n(L_n^+y)=\sum_{j=0}^{n}F^{j}(D_jy)+\sum_{j=1}^{n}\overline{D}_jF^{-j}y,
\end{align*}
where
\begin{equation}
  D_j(t)=\sum_{k=j}^{n}B_k(t+k-j)\overline{B}_{k-j}(t+k-j),\quad j=0,1,\ldots,n,\quad t\in\I.
\end{equation}
It is clear  $D_0(t)=\sum_{k=0}^{n}|B_k(t+k)|^2$  is real-valued on $\I$.
So, $L_nL_n^+$ is formally self-adjoint by Theorem 3.3.
The proof is complete.
\end{proof}

Now we can give another characterization of  formally self-adjoint difference expression $\mathcal{L}$ by using $L_n$.

\begin{theorem}
For any  formally self-adjoint difference expression $\mathcal{L}$ with order $2n$,
there exist two difference expressions  $L_n$ and $L_0$ with the form of $(3.3)$ such that
\begin{equation*}
  \mathcal{L}=L_nL_n^++L_0.
\end{equation*}
\end{theorem}

\begin{proof}
Let $\mathcal{L}$ be formally self-adjoint difference expression defined as (3.7).
By Theorem 3.4 and (3.8), it suffices to show there exists a set of complex-valued functions  $B_j(t)$, $0\le j\le n$, and
a real-valued function $C(t)$ such that
\begin{align}
    (L_ny)(t)=\sum_{j=0}^nF^{j}(B_jy)(t),\quad (L_0y)(t)=C(t)y(t)
\end{align}
and
\begin{align}
&\sum_{k=j}^{n} B_k(t+k-j)\overline{B}_{k-j}(t+k-j)\equiv A_j(t),\quad j=1,2,\ldots,n,\\
&\sum_{k=0}^{n} B_k(t+k-j)\overline{B}_{k-j}(t+k-j)+C(t)\equiv A_0(t).
\end{align}

By taking  $B_0(t)\equiv 1$ on $\I$, it can be easily verified that $B_n(t)= A_n(t)$ on $\I$ yields (3.10) with $j=n$.
In addition, by giving the following initial values
\begin{equation*}
  \begin{array}{lllll}
     B_1(0), &  &  &  &  \\
     B_2(0), &B_2(1),  &  &  &  \\
         &\cdots  &   &  & \\
     B_{n-1}(0),&B_{n-1}(1),&\cdots& B_{n-1}(n-2),
   \end{array}
\end{equation*}
the values  of $B_j(t)$, $t\ge j$, $1\le j\le n-1$,  can be determined uniquely by equation (3.10).
At last, take
\begin{equation*}
 C(t)=A_0(t)-\sum_{k=0}^{n} |B_k(t+k)|^2,\quad t\in \I.
\end{equation*}
Then  (3.11) is satisfied. The proof is complete.
\end{proof}

\section{Proof of Theorem 1.1}

Let $l(\I,n)$ be defined as that in Section 3,
and $\mathcal{L}$ be a formally self-adjoint difference expression of
order $2n$ defined  on $l(\I,n)$. Denote
\begin{align*}
l_{w}^2:=\left\{x\in l(\I,n):\;\sum_{t=0}^{+\infty}w(t)|x(t)|^2<+\infty\right\}.
\end{align*}
and
\begin{align*}
\langle x,y\rangle_{w}:=\sum_{t=0}^{+\infty}\bar{y}(t)w(t)x(t), \quad \|x\|_{w}:=\langle x,x\rangle_{w}^{1/2}.
\end{align*}
For any $x,y\in l_{w}^2$, we say $x=y$ if $\|x\|=\|y\|$.
Then $l_{w}^2$ is a Hilbert space with the inner product $\langle \cdot,\cdot\rangle_{w}$ (cf. \cite[Lemma 2.5]{Shi1}).

Denote
\begin{align*}
l^2_{w,0}(\I):=\left\{y\in l^2_{w}(\I):\; y(t)=0, \; -n\le t\le n-1\; {\rm and} \; t\ge k \;{\rm for\; some}\;k\in \I\right\}.
\end{align*}
and
\begin{eqnarray*}
&&T:=\{(x,f)\in (l^2_w(\I))^2: \;(\mathcal{L}x)(t)=w(t)f(t)\; {\rm on}\,\I\},\\
&&T_{0}:=\{(x,f)\in T:\; x\in l^2_{w,0}(\I)\}.
\end{eqnarray*}
$T$,  $T_{0}$, and  the closure of $T_{0}$, denoted by $\overline{T}_{0}$,  are called the maximal,
the  pre-minimal, and the  minimal linear relations corresponding to   $\mathcal{L}$, separately.
Thus, $d_{\ld}(\mathcal{L})$, given in Section 1, just refers to $d_{\ld}(\overline{T}_{0})$,
which is equal to $d_{\ld}(T_{0})$.

It follows from Theorem 3.5  there exist two difference expressions $L_n$ and $L_0$ defined as (3.9)
such that $\mathcal{L}=L_nL_n^++L_0$.
Similarly, denote
\begin{eqnarray*}
&&H:=\{(x,f)\in (l^2_w(\I))^2:\; (L_nL_n^+x)(t)=w(t)f(t)\; {\rm on}\;\I\},\\
&&H_{0}:=\{(x,f)\in H: x\in l^2_{w,0}(\I)\}.
\end{eqnarray*}

\begin{lemma}
Both $T_0$ and $H_{0}$ are Hermitian linear relations.
\end{lemma}

\begin{proof}
Since both $\mathcal{L}$ and $L_nL_n^+$ are formally self-adjoint,
it follows that both $T_0$ and $H_{0}$ are Hermitian linear relations by using (3.3).
\end{proof}

\begin{theorem} For  any formally self-adjoint difference expression $\mathcal{L}$,
\begin{equation}
 d_{+}(\mathcal{L})=d_{-}(\mathcal{L}).
\end{equation}
\end{theorem}

\begin{proof}
Based on the notations given above, (4.1) is equivalent to .
\begin{equation*}
 d_{+}(T_{0})=d_{-}(T_{0}).
\end{equation*}

For any $(x,f)\in H_0$, it follows  from (3.3) that
\begin{equation*}
  \langle x,f\rangle_w=\sum_{t\in \I} x(t)\overline{L_nL_n^+x(t)}=\sum_{t\in \I}  L_n^+x(t)\overline{L_n^+x(t)} \ge 0.
\end{equation*}
This yields that $H_0$ is Hermitian and bounded from below with lower bound $0$.
Hence, it follows from Lemma 2.5 that $H_0$ has self-adjoint extensions.
Let $H_1$ be a self-adjoint extension of $H_0$. By  $H_{1,s}$ denote the operator part of $H_1$.
It follows from Lemma 2.1 that $H_{1,s}$ is  self-adjoint in $(H_1^{\bot}(0))^2$,
and there exists uniquely a self-adjoint  operator, denoted by $S_{H}$ in $H_1^{\bot}(0)$
such that $G(S_{H})=H_{1,s}$.
Then by Theorem 2.8, there exists a conjugate $K$ on $H_{1}^{\bot}(0)$ such that $S_{H}$ is $K$-real.

Corresponding to difference expression $L_0$, we define
\begin{align*}
  &\mathcal{D}(S_0)=\{x\in l^2_w(\I)\cap H_{1}^{\bot}(0): \exists f\in l^2_w(\I)\;s.t.\; C(t)x(t)=w(t)f(t)\; {\rm on}\; \I \}, \\
  &S_0x=f.
\end{align*}
It can be easily verified that  $S_0$ is a well-defined operator on $l^2_w(\I)$.
Moreover,  $S_0$ is $K$-real.
In fact, for any $x\in \mathcal{D}(S_0)\cap H_{1}^{\bot}(0)$,
there exists uniquely  $f\in l^2_{w}(\I)$
such that $C(t)x(t)=w(t)f(t)$ on $\I$.
Since both $C(t)$ and $w(t)$ are real-valued,
it follows that
\begin{equation*}
  C(t)(Kx)(t)=w(t)(Kf)(t),\quad t\in \I.
\end{equation*}
This implies that $K\mathcal{D}(S_0)\subset \mathcal{D}(S_0)$ and  $S_0Kx=Kf=KS_0x$.
Thus $S_0$ is $K$-real, and consequently, $S_{H}+S_0$ is $K$-real.
So, by Lemma 2.7 one has that
\begin{equation*}
 d_+(S_{H}+S_0)=d_-(S_{H}+S_0),
\end{equation*}
 and consequently,
$S_{H}+S_0$ has self-adjoint extensions by Lemma 2.2.

On the other hand, it is clear $\overline{T}_0$ is closed and Hermitian.
Let  $\overline{T}_{0,s}$ be the operator part of $\overline{T}_{0}$ and
operator $S_T$ be the operator which satisfies  $G(S_T)=\overline{T}_{0,s}$.
Then one has   $S_{T}\subset S_{H}+S_0$.
Thus, all the self-adjoint extensions of  $S_{H}+S_0$ are still the self-adjoint extensions of  $S_{T}$.
Again by Lemma 2.2, one has $d_+(S_{T})=d_-(S_{T})$. This, together with $\mathcal{R}(\overline{T}_{0})=\mathcal{R}(S_{T})$,
yields  (4.1). The proof is complete.

\end{proof}

\noindent



\bibliographystyle{amsplain}

\end{document}